\documentclass[12pt]{amsart}
\usepackage{amssymb}
\usepackage{amscd}
\usepackage{mathrsfs}
\usepackage{amsrefs}
\usepackage{tikz}
\usepackage{color}
\usepackage{ulem}
\usepackage{comment}

\usepackage[margin=3cm]{geometry}
\usepackage{color}
\usepackage{empheq}
\usepackage{fancybox}

\makeatletter\tagsleft@false\makeatother

\theoremstyle{plain}
\newtheorem{thm}{Theorem}[section]

\newtheorem{lem}[thm]{Lemma}
\newtheorem{cor}[thm]{Corollary}
\theoremstyle{definition}
\newtheorem{defi}[thm]{Definition}
\newtheorem{rem}[thm]{Remark}
\newtheorem{ex}[thm]{Example}

\newcommand{\ZZ}{\mathbb{Z}}

\newcommand{\RR}{\mathbb{R}}
\newcommand{\CC}{\mathbb{C}}

\newcommand{\PP}{\mathbb{P}}

\newcommand{\ov}[1]{\overline{#1}}

\newcommand{\tl}[1]{\widetilde{#1}}

\newcommand{\Sing}{\mathrm{Sing}}

\newcommand{\Nt}{\mathrm{N}}

\newcommand{\grad}{\mathrm{grad}}
\newcommand{\Mil}{\mathrm{M}}
\newcommand{\Sim}{\mathrm{S}}

\newcommand{\Bif}{\mathrm{B}}
\newcommand{\IR}{\mathrm{I}}
\newcommand{\CF}{\mathrm{C}}

\newcommand{\naiseki}[1]{\left\langle #1\right\rangle}

\newcommand{\wt}[1]{\widetilde{#1}}

\newcommand{\Deg}{\mathrm{deg}}

\title[The bifurcation set of a rational function]
{The bifurcation set of a rational function via Newton polytopes}

\author[Tat Thang Nguyen]{{Tat Thang Nguyen}}
\author[Takahiro SAITO]{{Takahiro SAITO}}
\author[Kiyoshi Takeuchi]{{Kiyoshi Takeuchi}}

\address{Institute of Mathematics, Vietnam 
Academy of Science and Technology, 18 Hoang Quoc Viet Road,
Cau Giay District, Hanoi, Vietnam.}
\email{ntthang@math.ac.vn}
\address{Research Institute for Mathematical Sciences, Kyoto University, Kyoto 606-8502, Japan.}
\email{takahiro@kurims.kyoto-u.ac.jp}
\address{Mathematical Institute, Tohoku University,
Aramaki Aza-Aoba 6-3, Aobaku, Sendai, 980-8578, Japan.}
\email{takemicro@nifty.com}


\begin{document}

\maketitle

\begin{abstract}
The bifurcation sets of polynomial functions have been studied 
by many mathematicians from various points of view. 
In particular, N\'emethi and Zaharia described them 
in terms of Newton polytopes. In this paper, 
we will show analogous results for rational functions. 
\end{abstract}


\section{Introduction}\label{introduction}

Let $f(z) \in \CC [z_1, \ldots, z_n]$ be a
polynomial of $n$ ($\geq 2$) variables.
Then for the function $f \colon \CC^n
\longrightarrow \CC$ defined by it
there exists a finite
subset $B \subset \CC$
such that the restriction
\begin{align*}
\CC^n \setminus f^{-1}(B) \longrightarrow
\CC \setminus B
\end{align*}
of $f$ is a $C^{\infty}$ locally trivial fibration.
We denote by $\Bif_f$ the smallest
subset $B \subset \CC$ satisfying this
property. Let ${\rm Sing} f \subset \CC^n$
be the set of the critical points of
$f \colon \CC^n \longrightarrow \CC$. Then by
the definition of $\Bif_f$, obviously we have
\begin{align*}
f( {\rm Sing} f) \subset \Bif_f.
\end{align*}
The elements of $\Bif_f$ are called
bifurcation values of $f$.
The description of the bifurcation set
$\Bif_f \subset \CC$ is a fundamental problem and was
studied by many mathematicians e.g. \cite{Broughton},
\cite{C-D-T-T}, \cite{ChTib}, \cite{H-L}, \cite{H-N}, \cite{Is}, \cite{KOS}, \cite{N-Z1},
\cite{Nguyen}, \cite{P}, \cite{T}, \cite{Ta}, 
\cite{Tibar-book} 
and \cite{Zaharia} etc.
The essential difficulty lies in the
fact that in general $f$ has a
lot of singularities at infinity.
In \cite{N-Z1}, N\'emethi and Zaharia succeeded in
describing $\Bif_{f}$
in terms of the Newton polytope of $f$.
For the generalizations to polynomial maps
$f=(f_1, \ldots, f_k): \CC^n \rightarrow \CC^k$
for $n \geq k \geq 1$,
see \cite{C-D-T-T} and \cite{Nguyen}.
For the generalization to mixed polynomials, see \cite{ChTib}.

In this paper, we will show that analogous results
hold for rational functions.
Let $P(z), Q(z) \in \CC [z_1, \ldots, z_n]$ be
polynomials of $n$ ($\geq 2$) variables.
Assume that they are coprime each other.
Let
\begin{align*}
f(z)= \frac{P(z)}{Q(z)} \qquad
(z \in \CC^n \setminus Q^{-1}(0))
\end{align*}
be the rational function defined by them and
consider the map $f \colon \CC^n \setminus Q^{-1}(0)
\longrightarrow \CC$ associated to it.
Then as in the case of polynomial maps
we can define the bifurcation set
$\Bif_f \subset \CC$ of $f$ such that
$f( {\rm Sing} f) \subset \Bif_f$ (see
\cite{G-L-M}). After the pioneering paper
\cite{G-L-M} of Gusein-Zade, Luengo and
Melle-Hern\'andez, the local and global properties
of rational functions were studied from
various points of view by \cite{B-P},
\cite{B-P-S}, \cite{Thang}, \cite{N-T}
and \cite{Raibaut} etc.

In order to introduce our
main results, from now we prepare some notations.
Let $\Nt(P), \Nt(Q) \subset \RR_{\geq 0}^n$ be the
Newton polytopes of $P, Q$ respectively and
\begin{align*}
\Nt(f):=\Nt(P)+\Nt(Q)
\end{align*}
their Minkowski sum.
Recall that for a vector $u$ in the
dual vector space of $\RR^n$ we can
define its supporting faces in $\Nt(f)$,
$\Nt(P)$ and $\Nt(Q)$ (see Definition \ref{supface}
for the details). Then for a face $\gamma \prec \Nt(f)$
there exist faces $\gamma (P) \prec \Nt(P)$
and $\gamma (Q) \prec \Nt(Q)$ such that
\begin{align*}
\gamma = \gamma (P)+ \gamma (Q)
\end{align*}
(see Section \ref{prelim} for the details.).
We shall say that a face $\gamma \prec \Nt(f)$ is
of type I if it is supported by
a vector $u \in \RR^n \setminus \RR_{\geq 0}^n$
and the affice span
${\rm Aff}( \gamma (P)- \gamma (Q)) \simeq
\RR^{\dim \gamma}$ of the polytope
$\gamma (P)- \gamma (Q) \subset \RR^n$
in $\RR^n$ contains the origin $0 \in \RR^n$.
Clearly, if $Q(z)=1$ and $f(z)=P(z)$ is
a polynomial, this notion corresponds to
that of bad faces of $\Nt(f)= \Nt(P)$ defined by
N\'emethi and Zaharia \cite{N-Z1} (cf. \cite{Saito}, 
\cite{T} and \cite{T-T} for a slightly different one).
We denote the set of faces of $\Nt(f)$ of
type I by $\mathscr{F}_{I}$.
For $\gamma
\in \mathscr{F}_{I}$ by using the Laurent
polynomials $P_{\gamma (P)}(z)$ and
$Q_{\gamma (Q)}(z)$ on the torus
$T=( \CC^*)^n$ we define a function
$f_{\gamma}: T \setminus Q_{\gamma (Q)}^{-1}(0)
\longrightarrow \CC$ by
\begin{align*}
f_{\gamma}(z)=
\frac{P_{\gamma (P)}(z)}{Q_{\gamma (Q)}(z)}
\qquad
(z \in T \setminus Q_{\gamma (Q)}^{-1}(0))
\end{align*}
Then our main result is as follows.

\begin{thm}\label{MTM}
Assume that the divisor $P^{-1}(0)
\cup Q^{-1}(0) \subset \CC^n$ is
normal crossing in a neighborhood of
$P^{-1}(0) \cap Q^{-1}(0)$ and $f(z)= \frac{P(z)}{Q(z)}$ is
non-degenerate (see Definition \ref{nondeg}).
Then we have
\begin{align}\label{th11inc}
\Bif_f \subset f( {\rm Sing} f) \cup \{ 0 \}
\cup \Bigl(
\bigcup_{\gamma \in \mathscr{F}_{I}}
f_{\gamma}( {\rm Sing} f_{\gamma})
\Bigr).
\end{align}
\end{thm}

Note that the first assumption of this theorem is
satisfied by generic polynomials $P(x)$ and
$Q(x)$ such that $P(0) \not= 0$ and
$Q(0) \not= 0$. Moreover, in the two
dimensional case $n=2$ the same is true
also for generic $P(z)$ and $Q(z)$.
For $n \geq 2$, if the intersection
of $\Nt(Q)$ and each coordinate axis of $\RR^n$
is equal to $\{ 0 \} \subset \RR^n$ then
the the first assumption of Theorem \ref{MTM}
is satisfied by generic $P(z)$ and $Q(z)$.
Indeed, for such $Q(z)$ we have
\begin{align*}
Q^{-1}(0) \subset T=( \CC^*)^n \subset \CC^n.
\end{align*}
This is the case when $Q(z)=1$ and $f(z)=P(z)$
is a polynomial. If $f(z)=P(z)$ is non-degenerate
(at infinity) and convenient, by a result of
Broughton \cite{Broughton}
the polynomial map $f: \CC^n \rightarrow \CC$
is tame at infinity and
\begin{align*}
\Bif_f = f( {\rm Sing} f).
\end{align*}
However, for rational functions $f(z)= \frac{P(z)}{Q(z)}$,
by Theorem \ref{MTM} and the analogues of
the results in \cite{T} and \cite{Zaharia}
for rational functions (which can be proved
by toric compactifications of $\CC^n$), 
even if $P(z)$ and $Q(z)$ are convenient
there might be some type I faces of $\Nt(f)$ and
hence we do not have the equality
$\Bif_f = f( {\rm Sing} f)$ in general. 
See Section \ref{2-dim case} for the details. 
As in Gusein-Zade, Luengo
and Melle-Hern\'andez
\cite{G-L-M}, our non-degeneracy condition
in Definition \ref{nondeg} is inspired from the
classical one for polynomial functions over
complete intersection subvarieties in $\CC^n$
used by many authors such as
\cite{M-T-1} and \cite{Oka} etc. For the
proof of Theorem \ref{MTM} we also need to
refine the methods of N\'emethi and Zaharia in
\cite{N-Z1}. Finally, note that the monodromies
of rational functions over $\CC^n$ were
studied by \cite{G-L-M} and \cite{N-T}.

\section{Preliminary notions and results}\label{prelim}

Let $P(z), Q(z)\in \CC[z_{1},\dots,z_{n}]$ be 
polynomials of $n(\geq 2)$-variables with coefficients in $\CC$.
We define a rational function $f(z)$ by
\[f(z)=\dfrac{P(z)}{Q(z)} 
\qquad 
(z \in \CC^n \setminus Q^{-1}(0)). \] 
We will study the map from 
$\CC^n\setminus Q^{-1}(0)$ to $\CC$ defined by $f$. 
Let us set $\IR(f)=P^{-1}(0) \cap Q^{-1}(0) \subset \CC^n$. 
If $P$ and $Q$ are coprime, then 
$\IR(f)$ is nothing but the set of the indeterminacy 
points of $f$. In fact, 
the set $\IR(f)$ depends on the pair $(P(z), Q(z))$ of 
polynomials representing $f(z)$. 
For example, if we take a non-zero polynomial 
$R(z)$ on $\CC^n$ and set 
\begin{align*}
g(z)= \frac{P(z)R(z)}{Q(z)R(z)}  \qquad (z \in \CC^n), 
\end{align*}
then the set $\IR(g)=\IR(f) \cup R^{-1}(0)$ might be 
bigger than $\IR(f)$. In this way, we 
distinguish $f(z)= \frac{P(z)}{Q(z)}$ from 
$g(z)= \frac{P(z)R(z)}{Q(z)R(z)}$ even if 
their values coincide over an open dense 
subset of $\CC^n$. This is the convention due to 
Gusein-Zade, Luengo and Melle-Hern\'andez 
\cite{G-L-M} etc. Hereafter, we assume that 
$P(z)$ and $Q(z)$ are coprime. 

\begin{defi}
We say that $c\in \CC$ is an atypical 
value of $f$ if for any open neighborhood $U$ of $c$, 
the restriction $f^{-1}(U)\cap (\CC^n\setminus Q^{-1}(0))\to U$ of $f$
is not a $C^{\infty}$ trivial fibration. 
The bifurcation set $\Bif_{f}\subset \CC$ is 
the set of all the atypical values of $f$. 
\end{defi}

For a polynomial or rational function $g$ on $\CC^n$ as in \cite{Milnor},
we set 
\begin{align*}
\grad{g}(z):=\left(
\overline{\dfrac{\partial g}{\partial z_{1}}(z)},
\ldots, \ov{\dfrac{\partial g}{\partial z_{n}}(z)}\right), 
\end{align*}
where $\overline{a}$ is the complex conjugate of $a\in \CC$. 
For $z=(z_{1},\dots,z_{n}), w=(w_{1},\dots,w_{n})\in \CC^n$, 
$\langle z,w\rangle$ stands for the Hermite 
inner product of $z$ and $w$, i.e. 
$\langle z,w\rangle=\sum_{i=1}^{n}z_{i}\ov{w_{i}}$. 
Moreover, for $z \in \CC^n$ 
we set $\begin{Vmatrix}z\end{Vmatrix}:=\sqrt{\langle z,z\rangle}
\in \RR_{\geq 0}$.

\begin{defi}\label{defmf}
\begin{enumerate}
\item We define a subset $\Mil_f\subset \CC^n$
by
\begin{align*}
{\Mil_f:=\{z\in \CC^n\setminus Q^{-1}(0)\ |\ 
\mbox{there exists $\lambda\in \CC$ such that $\grad{f} (z)=\lambda z$}\}}
\end{align*}
\item We define a subset $\Sim_{f}\subset \CC$ by 
\begin{align*}
\Sim_f:=\left\{ s_0 \in \CC\ \middle|\ 
\begin{array}{l}
\mbox{there exists a sequence 
$\{z^{k}\}_{k=0}^{\infty}\subset \Mil_f$ such that}\\
\mbox{$\lim_{k\to \infty}\begin{Vmatrix}z^{k}
\end{Vmatrix}=\infty$ and $\lim_{k\to \infty}f(z^k)= s_0$.}
\end{array}
\right\}
\end{align*}
\end{enumerate}
\end{defi}
\begin{defi}
\begin{enumerate}
\item
Let $g(z)=\sum_{\alpha\in \ZZ^n}
a_{\alpha}z^{\alpha}\in \CC[z_{1}^{\pm},\dots,z_{n}^{\pm}]$ 
be a Laurent polynomial with coefficients in $\CC$.
Then the Newton polytope $\Nt(g)\subset \RR^n$ of $g$ 
is the convex full of the set $\mathrm{supp}(f)
:=\{\alpha\in \ZZ^n\ |\ a_{\alpha}\neq 0\}$ in $\RR^n$.
\item
Let $P(z), Q(z)\in \CC[z_{1},\dots,z_{n}]$ be 
polynomials and $f(z)$ the rational function 
$\frac{P(z)}{Q(z)}$ defined by them on $\CC^n$. 
Then the Newton polytope $\Nt(f)\subset \RR^n$ 
of $f$ is the Minkowski sum of $\Nt(P)$ and $\Nt(Q)$. 
Namely we set 
\[\Nt(f):=\{x+y\in \RR^n\ |\ x\in \Nt(P),\ y\in \Nt(Q)\}.\]
\end{enumerate}
\end{defi}

For real vectors $u=(u_{1},\dots,u_{n}), 
v=(v_{1},\dots,v_{n})\in \RR^n$,
we set $\langle u,v\rangle:=\sum_{i=1}^{n}u_{i}v_{i}$.

\begin{defi}\label{supface}
\begin{enumerate}
\item
Let $S$ be a polytope in $\RR^n$.
For a vector $u\in \RR^n$,
we set $d^{u}_{S}:=
\mathrm{min}_{w\in S}\langle u,w\rangle\in \RR$.
Moreover, 
for a real vector $u\in \RR^n$,
the supporting face $\gamma^{u}_{S}$ of $S$ by $u$ 
is a polytope defined by 
\[\gamma^{u}_{S}:=\{v\in S\ |\ 
\langle u,v\rangle = 
\mathrm{min}_{w\in S}\langle u,w\rangle
\}.\]
\item
For a Laurent polynomial 
$g(z)\in \CC[z_{1}^{\pm},\dots,z_{n}^{\pm}]$ and 
 a real vector $u\in \RR^n$ 
we set $d^{u}_{g}:=d^{u}_{\Nt(g)}$ and 
$\gamma^{u}_{g}:=\gamma^{u}_{\Nt(g)}$.
\item For a rational function $f(z)=\frac{P(z)}{Q(z)}$ on $\CC^n$ 
and a real vector $u\in \RR^n$, 
we set $d^{u}_{f}:=d^{u}_{P}-d^{u}_{Q}$ and 
$\gamma^{u}_{f}:=\gamma^{u}_{\Nt(f)}$.
\end{enumerate}
\end{defi}

\begin{defi}
Let $\frac{P(z)}{Q(z)}$ be a rational function on $\CC^n$.
\begin{enumerate}
\item
We say that a face $\gamma\prec \Nt(f)$ of $\Nt(f)$ is of type~I
if there exists a vector $u\in \RR^n\setminus 
\RR^{n}_{\geq 0}$ such that
$\gamma^{u}_{f}=\gamma$
and for any such $u$ we have $d_{f}^{u}(=d^{u}_{P}-d^{u}_{Q})=0$.
We denote the set of all the type~I faces of 
$\Nt(f)$ by $\mathscr{F}_{\mathrm{I}}$.
\item
We say that a face $\gamma\prec \Nt(f)$ of 
$\Nt(f)$ is of type~II, if it is not of type~I but there 
exists $u\in \RR^n\setminus \RR^n_{\geq 0}$ such that 
$\gamma^{u}_{f}=\gamma$. We denote the set of all the type~II faces of 
$\Nt(f)$ by $\mathscr{F}_{\mathrm{II}}$. 
\end{enumerate}
\end{defi}

For a Laurent polynomial $g(z)=
\sum_{\alpha\in \ZZ^n}a_{\alpha}z^{\alpha} 
\in \CC[z_{1}^{\pm},\dots,z_{n}^{\pm}]$ 
and a face $\gamma\prec \Nt(g)$, 
we set $g_{\gamma}(z):=\sum_{\alpha\in \gamma}a_{\alpha}z^{\alpha}$. 
We regard it as a function on $T=( \CC^*)^n$. 
Let $f(z)=\frac{P(z)}{Q(z)}$ be a rational function. Then for 
a face $\gamma\prec \Nt(f)$ and 
a real vector $u\in \RR^n$ such that 
$\gamma_f^u= \gamma$, the faces 
$\gamma_P^u \prec \Nt(P)$ and $\gamma_Q^u \prec \Nt(Q)$ 
do not depend on $u$. By taking such $u$ we set 
\begin{align*}
\gamma(P)= \gamma_P^u, \qquad 
\gamma(Q)= \gamma_Q^u. 
\end{align*}
Then we have 
\begin{align*}
\gamma = \gamma (P)+ \gamma (Q).
\end{align*}
Let 
${\rm Aff}( \gamma (P)- \gamma (Q)) \simeq
\RR^{\dim \gamma}$ be the affice span 
of the polytope 
$\gamma (P)- \gamma (Q) \subset \RR^n$
in $\RR^n$. Then the face 
$\gamma\prec \Nt(f)$ is of type~I 
iff it is supported by a vector $u\in \RR^n\setminus 
\RR^{n}_{\geq 0}$ and 
\begin{align*}
0 \in {\rm Aff}( \gamma (P)- \gamma (Q)). 
\end{align*}

\begin{ex}
Let the Newton polytopes of $P(z)$ and $Q(z)$ 
be as in Figures~\ref{fig1} and \ref{fig2}.
In this case, $\Nt(f)$ is a polytope as in Figure~\ref{fig3}. 
Then, the lines $\ov{OA}$, $\ov{OD}$ and $\ov{AB}$ and the points $O$ and $A$ are of type~I,
and the other faces of $\Nt(f)$ are of type~II.

\begin{figure}[htbp]
\begin{tabular}{ccc}
\begin{minipage}{0.3\hsize}
\begin{center}
\begin{tikzpicture}[scale = 0.5]
\draw[help lines] (-1,-1) grid (6,5);
\draw[fill,black, opacity=.5] (0,0) --(2,3)--(5,3)--(0,0)--cycle;
\draw[->,thick] (0,-1)--(0,5);
\draw[->,thick] (-1,0)--(6,0);
\draw (2.5,2) node{$\Nt(P)$};

\fill[black] (0,0) circle (0.1);
\fill[black] (2,3) circle (0.1);
\fill[black] (5,3) circle (0.1);

\end{tikzpicture}
\caption{$\Nt(P)$}\label{fig1}
\end{center}
\end{minipage}&
\begin{minipage}{0.3\hsize}
\begin{center}
\begin{tikzpicture}[scale = 0.5] 
\draw[help lines] (-1,-1) grid (6,5);
\draw[fill,black, opacity=.5] (0,0) --(2,3)--(4,1)--(0,0)--cycle;
\draw[->,thick] (0,-1)--(0,5);
\draw[->,thick] (-1,0)--(6,0);
\draw (2.2,1.5) node{$\Nt(Q)$};

\fill[black] (0,0) circle (0.1);
\fill[black] (2,3) circle (0.1);
\fill[black] (4,1) circle (0.1);

\end{tikzpicture}
\caption{$\Nt(Q)$}\label{fig2}
\end{center}
\end{minipage}&
\begin{minipage}{0.3\hsize}
\begin{center}
\begin{tikzpicture}[scale = 0.3] 
\draw[help lines] (-1,-1) grid (10,7);
\draw[fill,black, opacity=.5] (0,0) --(4,6)--(7,6)--(9,4)--(4,1)--(0,0)--cycle;
\draw[->,thick] (0,-1)--(0,7);
\draw[->,thick] (-1,0)--(10,0);
\draw (0,0) node[below left]{$O$};
\draw (4,6) node[above]{$A$};
\draw (7,6) node[above]{$B$};
\draw (9,4) node[right]{$C$};
\draw (4,1) node[below]{$D$};

\fill[black] (0,0) circle (0.1);
\fill[black] (4,6) circle (0.1);
\fill[black] (7,6) circle (0.1);
\fill[black] (9,4) circle (0.1);
\fill[black] (4,1) circle (0.1);

\draw (4.6,3) node{$\Nt(f)$};

\end{tikzpicture}
\caption{$\Nt(f)$}\label{fig3}
\end{center}
\end{minipage}
\end{tabular}
\end{figure}

\end{ex}

For a face $\gamma\prec \Nt(f)$ ($\gamma \not= \Nt(f)$) 
by using the Laurent
polynomials $P_{\gamma (P)}(z)$ and
$Q_{\gamma (Q)}(z)$ on the torus
$T=( \CC^*)^n$ we define a function
$f_{\gamma}: T \setminus Q_{\gamma (Q)}^{-1}(0)
\longrightarrow \CC$ by
\begin{align*}
f_{\gamma}(z)=
\frac{P_{\gamma (P)}(z)}{Q_{\gamma (Q)}(z)} 
\qquad 
(z \in T \setminus Q_{\gamma (Q)}^{-1}(0))
\end{align*}

\begin{defi}\label{nondeg}
Let $f(z)=\frac{P(z)}{Q(z)}$ be a rational function on $\CC^n$. 
Then we say that $f$ is non-degenerate if 
$\grad P_{\gamma(P)}(z)$ (resp. $\grad Q_{\gamma(Q)}(z)$) 
does not vanish on $P^{-1}_{\gamma(P)}(0)\setminus Q^{-1}_{\gamma(Q)}(0) 
\subset T$ (resp. $Q^{-1}_{\gamma(Q)}(0)\setminus P^{-1}_{\gamma(P)}(0) 
\subset T$) 
for any face $\gamma\prec \Nt(f)$ of type~II,
and the two vectors $\grad P_{\gamma(P)}(z)$ and 
$\grad Q_{\gamma(Q)}(z)$ are linearly independent 
on $P^{-1}_{\gamma(P)}(0)\cap Q^{-1}_{\gamma(Q)}(0) \subset T$
for any face $\gamma\prec \Nt(f)$ of type~I or II.
\end{defi}

\begin{lem}\label{type2sing}
Let $f(z)=\frac{P(z)}{Q(z)}$ be a rational function.
Assume that a face $\gamma\prec \Nt(f)$ is of type~II. 
Then we have $f_{\gamma}( \Sing{f_{\gamma}})\subset \{0\}$.
If moreover $f$ is non-degenerate in the sense of Definition~\ref{nondeg},
we have $f_{\gamma}( \Sing{f_{\gamma}})=\emptyset$.
\end{lem}

\begin{proof}
By the definition of faces of type~II, 
we can take a vector $u=(u_{1},\dots,u_{n}) 
\in \RR^n\setminus \RR^n_{\geq 0}$ 
such that $\gamma^{u}_{f}=\gamma$ and 
$d_{f}^{u}= d_{P}^{u}-d_{Q}^{u} \neq 0$. 
To prove the first assertion, assume that there exists non-zero $t_{0}\in f_{\gamma}(\Sing{f_{\gamma}})$,
i.e. there is a point $z^{0}\in \Sing{f_{\gamma}}$ such that $f_{\gamma}(z^{0})=t_{0}\neq 0$.
Since we have
\[\dfrac{\partial f_{\gamma}}{\partial z_{i}}(z^{0})=
\left(\dfrac{\partial P_{\gamma(P)}}{\partial z_{i}}(z^{0})-t_{0}\dfrac{\partial Q_{\gamma(Q)}}{\partial z_{i}}(z^{0})\right)\cdot \dfrac{1}{Q_{\gamma(Q)}(z^{0})}\quad (i=1,\dots, n),\]
we obtain 
\[\dfrac{\partial P_{\gamma(P)}}{\partial z_{i}}(z^{0})=t_{0}\cdot \dfrac{\partial Q_{\gamma(Q)}}{\partial z_{i}}(z^{0})\quad (i=1,\dots,n).\]
By Euler's theorem for quasi-homogeneous polynomials, we get
\[d^{u}_{P}\cdot P_{\gamma(P)}(z^{0})=t_{0}\cdot d_{Q}^{u}\cdot Q_{\gamma(Q)}(z^{0}).\] 
Since $f_{\gamma}(z^{0})=t_{0}$ and $t_{0}\neq 0$,
we have
\[d_{P}^{u}=d_{Q}^{u},\]
which is a contradiction.
The second assertion is now clear since if $f$ 
is non-degenerate, the central fiber $f^{-1}_{\gamma}(0)=
P^{-1}_{\gamma(P)}(0)\setminus Q^{-1}_{\gamma(Q)}(0)$ is also smooth.

\end{proof}

We will use the following lemma.

\begin{lem}[{Curve Selection Lemma, c.f. 
{\cite[Lemma~2]{N-Z2}}}]\label{cvsel}
Let $f_{1}(x),\dots, f_{s}(x)$, $g_{1}(x),\dots, 
g_{t}(x)$, $h_{1}(z),\dots,h_{u}(x)\in \RR[x_{1},\dots,x_{n}]$
be polynomials with real coefficients.
Let $U=\{x\in \RR^m\ |\ f_{i}(x)=0,\ 1\leq i \leq s\}$
and
$W=\{x\in \RR^m\ |\ g_{i}(x)>0,\ 1\leq 
i\leq t\}$.
Suppose that there exists a sequence 
$\{x^{k}\}_{k=0}^{\infty}\subset U\cap W$ such that 
$\lim_{k\to \infty}\begin{Vmatrix}x^{k}\end{Vmatrix}=\infty$
and for all $1\leq i\leq u$, $\lim_{k\to \infty}h_{i}(x^{k})=0$.
Then, there exists a real analytic curve $p\colon (0,1)\to U\cap W$
of the form $p(t)=at^{\alpha}+a_{1}
t^{\alpha+1}+\dots$ with $a\in \RR^m\setminus \{0\}$ and $\alpha<0$
such that $\lim_{t\to 0}\begin{Vmatrix}p(t)\end{Vmatrix}=\infty$ and 
$\lim_{t\to 0}h_{i}(p(t))=0$ for any $1\leq i\leq u$.
\end{lem}

\begin{rem}
By the proof of the above lemma in \cite{N-Z2},
we see moreover that 
$\alpha$ is a half integer.
\end{rem}

\section{Main theorems}

\begin{thm}\label{NST}
Let $f(z)=\frac{P(z)}{Q(z)}$ be a
rational function $\CC^n\setminus Q^{-1}(0)\to \CC$.
Assume that the divisor $P^{-1}(0)
\cup Q^{-1}(0) \subset \CC^n$ is
normal crossing in a neighborhood of
$P^{-1}(0) \cap Q^{-1}(0)$.
Then we have
\begin{align*}
\Bif_f \subset f({\rm Sing} f) \cup \Sim_{f}.
\end{align*}
\end{thm}

\begin{proof}
First, let us consider the simplest case where
$P^{-1}(0), Q^{-1}(0) \subset \CC^n$ are smooth
and intersect transversally. For $R>0$ we set
$S_R= \{ z \in \CC^n | \begin{Vmatrix}z\end{Vmatrix} =R \}$.
Let $\mathcal{S}$ be the coarsest Whitney stratification
of the normal crossing divisor $P^{-1}(0)
\cup Q^{-1}(0)$. Then there exists $R_0 \gg 0$
such that for any $R>R_0$ the sphere $S_R$
intersects each stratum in $\mathcal{S}$ transversally.
Now let $s_0 \in \CC$ be a point
such that $s_0 \notin f({\rm Sing} f) \cup \Sim_{f}$ and
$D \subset \CC$ a small open disc centered at $s_0$
satisfying the condition
\begin{align*}
\overline{D} \subset \CC \setminus \{
f({\rm Sing} f) \cup \Sim_{f} \}.
\end{align*}
Then by an analogue of
N\'emethi and Zaharia \cite[Lemma 3]{N-Z1} for
rational functions, there exists $R_1 \geq R_0$
such that
\begin{align*}
f^{-1}(D) \cap \Mil_f \cap
\{ z \in \CC^n | \begin{Vmatrix}z\end{Vmatrix} > R_1 \} = \emptyset.
\end{align*}
This implies that for any $R>R_1$ the sphere $S_R$
intersects the fiber $f^{-1}(s)$ transversally
for any $s \in D$. Let $\pi : \widetilde{\CC^n}
\rightarrow \CC^n$ be the blow-up of $\CC^n$
along $P^{-1}(0) \cap Q^{-1}(0)$ and
$E= \pi^{-1} \{ P^{-1}(0) \cap Q^{-1}(0) \}$
the exceptional divisor in it. Then the
meromorphic extension $g:=f \circ \pi$ of $f$
to $\widetilde{\CC^n}$ has no point of
indeterminacy and for any $s \in \CC$ its
fiber $g^{-1}(s)$ intersects $E$ transversally.
Moreover for $R>R_0$ we see that the closure
\begin{align*}
\widetilde{S_R}:= \overline{
\pi^{-1} [ S_R \setminus
\{ P^{-1}(0) \cap Q^{-1}(0) \} ] }
\subset \widetilde{\CC^n}
\end{align*}
is a smooth real hypersurface of the complex
manifold $\widetilde{\CC^n}$. For $s \in \CC$
let $\mathcal{S}_s$ be the coarsest Whitney stratification
of the normal crossing divisor $g^{-1}(s)
\cup E$. Then for any $R>R_1$ the real hypersurface
$\widetilde{S_R}$ intersects each stratum in
$\mathcal{S}_s$ transversally. This implies that
for any point of $g^{-1}(s) \cap E \cap \widetilde{S_R}$
and a local coordinate system $\zeta =( \zeta_1, \zeta_2, \ldots,
\zeta_n)$ of $\widetilde{\CC^n}$ around it
such that $E= \{ \zeta_1=0 \}$ we can find locally a
smooth real vector field $v( \zeta )$ on
$\widetilde{\CC^n}$ such that
\begin{align*}
v( \zeta ) \zeta_1 \equiv 0, \qquad
v( \zeta ) g( \zeta ) \equiv 1
\end{align*}
and $v( \zeta )$ is tangent to the real
hypersurface
$\widetilde{S_{ \begin{Vmatrix} \pi ( \zeta ) \end{Vmatrix} }}$
passing through the point $\zeta$.
By the first (resp. third) condition on $v( \zeta )$,
its integral curves do not go into the exceptional
divisor $E$ (resp. at infinity) in finite time.
Now by
our choice of $D$ and the construction of
the blow-up $\pi$, the morphisim
$g^{-1}(D) \rightarrow D$ induced by $g$
is a (non-proper)
holomorphic submersion. Moreover
the boundary of
the closure
\begin{align*}
\overline{g^{-1}(D)}= \overline{
\pi^{-1} f^{-1} (D) }
\subset \widetilde{\CC^n}
\end{align*}
is smooth and intersects $E$ transversally.
Then as in the proof of
N\'emethi and Zaharia \cite[Theorem 1]{N-Z1},
by using a partition of unity we can construct
a smooth real vector field $\tilde{v}$ globally defined
on $g^{-1}(D)$ such that
\begin{align*}
\tilde{v} g \equiv 1
\end{align*}
whose integral curves do not go into the exceptional
divisor $E$ or at infinity in finite time.
By the restriction $u$ of $\tilde{v}$ to
$f^{-1}(D) = g^{-1}(D) \setminus E \subset \CC^n$
and its multiple $iu$ ($i:= \sqrt{-1}$) we can
prove that the
morphism $f^{-1}(D) \rightarrow D$ is a
$C^{\infty}$ trivial fibration over $D$. Finally,
let us consider the general case.
We can construct a composition
$\pi : \widetilde{\CC^n}
\rightarrow \CC^n$ of several blow-ups
of $\CC^n$ over $P^{-1}(0) \cap Q^{-1}(0)$
so that the
meromorphic extension $g:=f \circ \pi$ of $f$
to $\widetilde{\CC^n}$ has no point of
indeterminacy (see e.g.
the proof of \cite[Theorem 3.6]{M-T-2} and 
\cite[Section~3]{M-T-3}).
Then the proof proceeds similarly to
the one in the previous case.
This completes the proof.
\end{proof}

Note that the assumption of this theorem are
satisfied by generic polynomials $P(z)$ and
$Q(z)$ such that $P(0) \not= 0$ and
$Q(0) \not= 0$. Moreover, in the two
dimensional case $n=2$ the same is true
also for generic $P(z)$ and $Q(z)$.
For $n \geq 2$, if the intersection
of $\Nt(Q)$ and each coordinate axis of $\RR^n$
is equal to $\{ 0 \} \subset \RR^n$ then
the assumption of Theorem \ref{NST}
is satisfied by generic $P(z)$ and $Q(z)$.
Indeed, for such $Q(z)$ we have
\begin{align*}
Q^{-1}(0) \subset T=( \CC^*)^n \subset \CC^n.
\end{align*}
This is the case when $Q(z)=1$ and $f(z)=P(z)$
is a polynomial.

\begin{thm}\label{main2}
Let $f(z)=\frac{P(z)}{Q(z)}$ be a rational 
function $\CC^n\setminus Q^{-1}(0)\to \CC$.
Assume that $f$ is non-degenerate in the 
sense of Definition~\ref{nondeg}.
Then, we have
\begin{align}\label{main2cont}
\Sim_{f} \subset  \{0\}\cup 
\Bigl( \bigcup_{\gamma\in 
\mathscr{F}_{\mathrm{I}}}f_{\gamma}(\Sing{f_{\gamma}})
\Bigr).
\end{align}
\end{thm}

\begin{proof}
Our proof is inspired from that 
of \cite[Theorem 2]{N-Z1}. 
Assume that $s_0 \in \Sim_{f}$.
Then, by the definition of $\Sim_{f}$, 
there exists a sequence 
$\{z^{k}\}_{k=0}^{\infty}$ in $\Mil_{f}$ such that 
$\lim_{k\to \infty}\begin{Vmatrix}z^{k}\end{Vmatrix}
=\infty$ and $\lim_{k\to \infty}f(z^{k})= s_0$.
By the curve selection lemma (Lemma~\ref{cvsel}), 
we can take an analytic curve $h(t)\colon (0,1)\to \CC^n$ of the form
\begin{align}\label{hexp}
h(t)=at^{\alpha}+a_{1}t^{\alpha+1}+
\cdots\quad (\mbox{$a\neq 0$ and $\alpha<0$}),
\end{align}
satisfying the conditions:
\begingroup
\renewcommand{\arraystretch}{1.6}
\[
\left\{\begin{array}{l}
h(t)\in \Mil_f \quad (t\in (0,1)),\\
\lim_{t\to 0}
\begingroup
\renewcommand{\arraystretch}{1}
\begin{Vmatrix}h(t)\end{Vmatrix}=\infty, 
\endgroup
\\
\lim_{t\to 0}f(h(t))=s_0.
\end{array}\right.
\]
\endgroup
By the definition of $\Mil_{f}$, 
there is an analytic function $\lambda(t)
\colon (0,1)\to \CC$ such that 
\begin{align}\label{iden3}
\grad{f}(h(t))=\lambda(t)h(t).
\end{align}
We will use the identities:
\begin{align}\label{iden1}
\dfrac{df(h(t))}{dt}=\naiseki{\dfrac{dh}{dt}(t),
\grad{f}(h(t))}.
\end{align}
If $\grad{f}(h(t))\equiv 0\ (t\in (0,1))$, the 
identity (\ref{iden1}) implies that
$\frac{df(h(t))}{dt}\equiv 0$ and $f(h(t))$ is 
a constant function.
Hence $\sigma=\lim_{t\to 0}f(h(t))$ is in $\mathrm{Sing}{f}$.
Therefore, we can assume $\grad{f}(h(t))
\not\equiv 0$.
If $f(h(t))\equiv 0$, the identities (\ref{iden1}) 
and (\ref{iden3}) imply that
\[\ov{\lambda(t)}\naiseki{\dfrac{dh}{dt}(t),h(t)}\equiv 0.\]
By (\ref{hexp}), we have 
\[\naiseki{\dfrac{dh}{dt}(t),h(t)}=|a|^2\alpha t^{2\alpha-1}+\cdots.\]
Here $\cdots$ stands for higher order terms. 
In particular, $\naiseki{\frac{dh}{dt}(t),h(t)}\not\equiv 0$ 
and we thus obtain
$\lambda(t)\equiv 0$, which is in contradiction 
with $\grad{f}(h(t))(=\lambda(t)h(t))\not\equiv 0$.
So, we will also assume $f(h(t))\not\equiv 0$.

Let the expansions of $f(h(t)), \grad(f(h(t)))$ and 
$\lambda(t)$ be of the following forms:
\begin{empheq}[left=\empheqlbrace]{align}\notag
f(h(t))=bt^{\beta}+ \cdots \cdots,\\
\grad{f}(h(t))=ct^{\rho}+ \cdots \cdots,\notag\\
\lambda(t)=\lambda_{0} t^{\delta}+ \cdots 
\cdots,\notag
\end{empheq}
where $b \in \CC, c \in \CC^n, \lambda_{0} \in \CC$ are not zero. 
Note that the assumption $\lim_{t\to 0}f(h(t))=s_0
\in \CC$ implies $\beta\geq 0$.
By considering the expansions of both sides of 
(\ref{iden3}), we have
\begin{align*}
\rho&=\delta+\alpha, \quad \mbox{and}\\
c&=\lambda_{0}a.
\end{align*}
Hence, we have $\naiseki{a,c}\neq 0$. 
For an analytic function $g(t)=g_{0}t^{\eta}+
\cdots \cdots$ ($g_0 \not= 0$), 
we denote by $\Deg{g(t)}$ its degree with respect to $t$. 
Namely we set $\Deg{g(t)}=\eta$. 
Then the degree of the right hand side of 
(\ref{iden1}) is equal to $\alpha-1+\rho$.
By (\ref{iden1}), we thus obtain 
\[\alpha-1+\rho\ (=\beta-1)\geq -1,\]
which implies $\rho>0$ since we have $\alpha<0$.
Moreover, we have 
\[\delta=\rho-\alpha>0.\]

We may assume that 
\begin{align}\label{ht}
h(t)=(w_{1}^{0}t^{\nu_{1}}+w_{1}^{1}t^{\nu_{1}+1}+
\cdots, \dots, w_{k}^{0}t^{\nu_{k}}+w_{k}^{1}t^{\nu_{k}+1}+\cdots,0,\dots,0),
\end{align}
where $w_{1}^{0}\neq 0,\dots, w_{k}^{0}\neq 0$ and 
$\alpha=\nu_{1}\leq \nu_{2}\leq \dots \leq \nu_{k}$.
We identify 
\[\{(x_{1},\dots,x_{n})\in \RR^n\ |\ x_{k+1}=\dots=x_{n}=0\}\]
with $\RR^{k}$.
Then, we will consider the supporting face 
$\gamma\subset \RR^k(\subset \RR^n)$ of 
$\Nt(f)\cap \RR^k(=\Nt(P)\cap \RR^k+\Nt(Q)\cap \RR^k)$ 
by the vector $(\nu_{1},\dots, 
\nu_{k})\in \RR^k$.
Since $f(h(t))\not\equiv 0$, we have 
$\Nt(P)\cap \RR^k\neq \emptyset$ and 
$\Nt(Q)\cap \RR^k\neq \emptyset$.
Let $m(<0)$ be a real number smaller than 
the (non-positive) integer 
\begin{align*}
\min\{\nu_{1}w_{1}+\dots+\nu_{k}w_{k}\in 
\RR \ |\ (w_{1},\dots,w_{n})\in \Nt(f)\}&
\\-\min\{\nu_{1}w_{1}+\dots+\nu_{k}w_{k}\in 
\RR &\ |\ (w_{1},\dots,w_{k})\in \Nt(f)\cap \RR^k\} 
\end{align*}
and set 
\begin{align*}
\nu:=(\nu_{1},\dots,\nu_{k}, -m,\dots,-m) \in \RR^n. 
\end{align*}
Then $\gamma$ is the supporting face of 
$\Nt(f)(\subset \RR_{\geq 0}^{n})$ by 
$\nu \in \RR^n$. 
Recall that by using the decomposition $\gamma=\gamma(P)+
\gamma(Q)\ (\gamma(P)\prec \Nt(P), \gamma(Q)\prec \Nt(Q))$ 
we defined $f_{\gamma}(z):=\frac{P_{\gamma(P)}(z)}{Q_{\gamma(P)}(z)}$ 
and $d_{f}^{\nu}=d_{P}^{\nu}-d_{Q}^{\nu}$.
Set 
\[w^{0}:=(w_{1}^{0},\dots,w_{k}^{0},1,\dots,1)\in T=( \CC^*)^n.\]
Then, for $j=1,\dots,k$ we have

\begin{empheq}[left=\empheqlbrace]{align}\notag
P(h(t))&=P_{\gamma(P)}(w^{0})t^{d_{P}^{\nu}}+\cdots,\\
\dfrac{\partial P}{\partial z_{j}}(h(t))&=
\dfrac{\partial P_{\gamma(P)}}{\partial z_{j}}(
w^{0})t^{d_{P}^{\nu}-\nu_{j}}+\cdots,\notag
\end{empheq}
\begin{empheq}[left=\empheqlbrace]{align}\notag
Q(h(t))&=Q_{\gamma(Q)}(w^{0})t^{d_{Q}^{\nu}}+\cdots,\\
\dfrac{\partial Q}{\partial z_{j}}(h(t))&=\dfrac{\partial 
Q_{\gamma(Q)}}{\partial z_{j}}(w^{0})
t^{d_{Q}^{\nu}-\nu_{j}}+\cdots.\notag
\end{empheq}
We set 
\begin{empheq}[left=\empheqlbrace]{align}\notag
e_{P}&:=\Deg{P(h(t))},\\
e_{Q}&:=\Deg{Q(h(t))}.\notag 
\end{empheq}
Namely the expansions of $P(h(t))$ and $Q(h(t))$ are of the form: 
\begingroup
\renewcommand{\arraystretch}{1.6}
\[
\left\{\begin{array}{l}
P(h(t))=P_{e_{P}}t^{e_{P}}+ \cdots \cdots, \\
Q(h(t))=Q_{e_{Q}}t^{e_{Q}}+ \cdots \cdots,
\end{array}
\right.
\]
\endgroup
with $P_{e_{P}}\neq 0$ and $Q_{e_{P}}\neq 0$. 
Note that 
\begin{align}\label{bigstar}\tag{$\bigstar$}
\left\{\begin{array}{l}
e_{P}\geq d_{P}^{\nu}, \\
e_{Q}\geq d_{Q}^{\nu}.
\end{array}
\right.
\end{align}
Since $\lim_{t\to 0}f(h(t))=
s_{0}\in \CC$, we have $e_{P}\geq e_{Q}$.
If $e_{P}>e_{Q}$, the value $s_0=\lim_{t\to 0}f(h(t))$ is $0$
and contained in the right hand side of (\ref{main2cont}).
So we will assume $e:=e_{P}=e_{Q}$ in the following.

We set 
\[l:=\min\{d_{P}^{\nu}, d_{Q}^{\nu}\}.\]
We will use the obvious identity:
\begin{align}\label{iden4}
\ov{Q(h(t))}\grad{P}(h(t))-\ov{P(h(t))}
\grad{Q}(h(t))=\ov{Q^2(h(t))}\grad{f}(h(t)).
\end{align}
By (\ref{iden3}) and (\ref{ht}), 
the $j(>k)$-th entry of the right hand side of (\ref{iden4}) is zero. 
Note also that
for $1\leq j\leq k$ the degree of the $j$-th entry 
 of the left hand side of (\ref{iden4}) 
is larger than or equal to $e+l-\nu_{j}$.
We set 
\begingroup
\renewcommand{\arraystretch}{1.6}
\begin{align}\label{pti}
\wt{P_{e}}&:=\left\{
\begin{array}{ll}
P_{e}&(\mbox{if $l=d_{Q}^{\nu}$})\\
0&(\mbox{otherwise}).
\end{array}
\right.
\\
\wt{Q_{e}}&:=\left\{
\begin{array}{ll}
Q_{e}&(\mbox{if $l=d_{P}^{\nu}$})\\
0&(\mbox{otherwise}).
\end{array}
\right.\label{qti}
\end{align}
\endgroup
Note that at least one of $\wt{P_{e}}$ and 
$\wt{Q_{e}}$ is not zero. 
For $1 \leq j \leq k$ let $A_{j} \in \CC$ be the 
coefficient of $t^{e+l-\nu_{j}}$ in the $j$-th entry of the 
left hand side of (\ref{iden4}). 
Then its complex conjugate $\ov{A_{j}}$ is expressed as 
\[
\ov{A_{j}}=\wt{Q_{e}}
\dfrac{\partial 
P_{\gamma(P)}}{\partial z_{j}}(w^{0})-
\wt{P_{e}}\dfrac{\partial Q_{\gamma(Q)}}{\partial z_{j}}(w^{0}).
\]
Namely we have 
\begin{align}\label{key}
\left(
    \begin{array}{c}
      \ov{A_{1}} \\
      \ov{A_{2}} \\
      \vdots \\
      \ov{A_{k}}
    \end{array}
  \right)
= \wt{Q_{e}}
\left(
    \begin{array}{c}
      \frac{\partial P_{\gamma(P)}}{\partial z_{1}}(w^{0}) \\
      \frac{\partial P_{\gamma(P)}}{\partial z_{2}}(w^{0}) \\
      \vdots \\
      \frac{\partial P_{\gamma(P)}}{\partial z_{k}}(w^{0})
    \end{array}
  \right)
- \wt{P_{e}}
\left(
    \begin{array}{c}
      \frac{\partial Q_{\gamma(Q)}}{\partial z_{1}}(w^{0}) \\
      \frac{\partial Q_{\gamma(Q)}}{\partial z_{2}}(w^{0}) \\
      \vdots \\
      \frac{\partial Q_{\gamma(Q)}}{\partial z_{k}}(w^{0})
    \end{array}
  \right).
\end{align}
We set 
\begin{align}\label{ai}
J:=\{1\leq j\leq k\ |\ A_{j}\neq 0\},\mbox{\ and}\\
j_{0}:=\min J \mbox{\ (when $J\neq \emptyset$)}.\label{jzero}
\end{align}
If $J \not= \emptyset$ and 
$A_{j}\neq 0$ for $j \in J$, by (\ref{iden3}) and (\ref{iden4}),
we have
\begin{align}\label{jisuu}
e+l-\nu_{j}=&2e+\delta+\nu_{j},\mbox{\ and}\\
\ov{A_{j}}=&Q_{e}^2\ov{\lambda_{0}}\ov{w^{0}_{j}}.\label{cj}
\end{align}
Therefore, we have 
\[\nu_{j}=\dfrac{1}{2}(-e+l-\delta)\]
and in particular
$\nu_{j}=\nu_{j_{0}}\ (j\in J)$.
Moreover, since $e\geq l$ and $\delta>0$,
we have 
\begin{align}
\nu_{j}<0 \label{nuineq}
\end{align}
for such $j$.

\begin{lem}\label{NZlem5}
If $J\neq \emptyset$, we have the equality
\begin{align}\label{iden5}
Q_{e}^2\nu_{j_{0}}\ov{\lambda_{0}}\sum_{j\in J}|w^{0}_{j}|^2=
\wt{Q_{e}}d_{P}^{\nu}P_{\gamma(P)}(w^{0})
-\wt{P_{e}}d_{Q}^{\nu}Q_{\gamma(Q)}(w^{0}).
\end{align}
In particular, the right hand side of (\ref{iden5}) is not $0$.
\end{lem}

\begin{proof}[Proof of Lemma~\ref{NZlem5}]
Assume $J\neq \emptyset$. 
By Euler's equality for quasi-homogeneous polynomials, we have
\begin{align}\label{EulP}
\sum_{1\leq j\leq k}\nu_{j}w^{0}_{j}
\dfrac{\partial P_{\gamma(P)}}{\partial z_{j}}(w^{0})
=&d_{P}^{\nu}P_{\gamma(P)}(w^{0}), \mbox{\ and}\\
\sum_{1\leq j\leq k}\nu_{j}w^{0}_{j}
\dfrac{\partial Q_{\gamma(Q)}}{\partial z_{j}}(w^{0})
=&d_{Q}^{\nu}Q_{\gamma(Q)}(w^{0}).\label{EulQ}
\end{align}
Then we have
\begin{align}\sum_{j\in J}w_{j}^{0}\nu_{j}\ov{A_{j}}
=&\sum_{1\leq j\leq k}w_{j}^{0}\nu_{j}\ov{A_{j}}\notag\\
=&\sum_{1\leq j\leq k}w_{j}^{0}\nu_{j} \Bigl\{ 
\wt{Q_{e}}\dfrac{\partial P_{\gamma(P)}}{\partial 
z_{j}}(w^{0})-\wt{P_{e}}\dfrac{\partial 
Q_{\gamma(Q)}}{\partial z_{j}}(w^{0}) \Bigr\} \notag \\
=&\wt{Q_{e}}d_{P}^{\nu}P_{\gamma(P)}(w^{0})-
\wt{P_{e}}d_{Q}^{\nu}Q_{\gamma(Q)}(w^{0}) 
\quad (\mbox{by (\ref{EulP}) and (\ref{EulQ})}). \label{sumI1}
\end{align}
On the other hand,
by (\ref{cj}), we have
\begin{align}\label{sumI2}
\sum_{j\in J}w_{j}^{0}\nu_{j}\ov{A_{j}}=
Q_{e}^2\nu_{j_{0}}\ov{\lambda_{0}}\sum_{j\in J}|w^{0}_{j}|^2.
\end{align}
Combining (\ref{sumI1}) and (\ref{sumI2}),
we obtain the desired equality.

The second assertion follows from the facts: 
$Q_{e}\neq 0$, $\lambda_{0}\neq 0$, $w^{0}_{j}\neq 0$ and (\ref{nuineq}).
\end{proof}

Now, let us finish the proof of Theorem~\ref{main2}.
\\\\
\noindent (\textbf{Case 1}) We first 
assume that $P_{\gamma(P)}(w^{0})\neq 0$ and $Q_{\gamma(Q)}(w^{0})\neq 0$.
In this case,
we have $e=e_{P}=d_{P}^{\nu}$ and $e=e_{Q}=d_{Q}^{\nu}$, and hence
\[l=d_{P}^{\nu}=d_{Q}^{\nu} \mbox{\ and\ } d_{f}^{\nu}=0.\]
Therefore, we have
\begin{align*}
(\mbox{RHS\ of (\ref{iden5})})
&=d_{P}^{\nu} \Bigl\{ \wt{Q_{e}}P_{\gamma(P)}(w^{0})
-\wt{P_{e}}Q_{\gamma(Q)}(w^{0}) \Bigr\} 
 \mbox{\quad (since $d_{P}^{\nu}=d_{Q}^{\nu}$)}\\
&=0 \mbox{\quad (since $\wt{P_{e}}=P_{\gamma(P)}
(w^{0})$ and $\wt{Q_{e}}=Q_{\gamma(Q)}(w^{0})$)}.
\end{align*}
If $J\neq \emptyset$,
this contradicts the second assertion of Lemma~\ref{NZlem5}.
Therefore, 
we have $J= \emptyset$ i.e. $A_{j}=0\ (1\leq j\leq k)$.
Moreover, for $1\leq j\leq k$,
we have
\[\ov{A_{j}}=Q_{\gamma(Q)}(w_{0})\dfrac{\partial 
P_{\gamma(P)}}{\partial z_{j}}(w^{0})-P_{\gamma(P)}(
w^{0})\dfrac{\partial Q_{\gamma(Q)}}{\partial 
z_{j}}(w^{0}), \mbox{\ and hence}
\]
\[\dfrac{\partial f_{\gamma}}{\partial z_{j}}(w^{0})=
\dfrac{\ov{A_{j}}}{{Q_{\gamma(Q)}^2(w^{0})}}=0.\]
Therefore, we have $w^{0}\in \Sing{f_{\gamma}}$.
Since $\nu \in \RR^n \setminus \RR^n_{\geq 0}$ the 
face $\gamma$ is of type I or II. But 
Lemma~\ref{type2sing} implies 
that $\gamma$ is of 
type~I and hence 
\begin{align*}
s_0 = \lim_{t\to 0}f(h(t))
= f_{\gamma}(w^0)\in f_{\gamma}(\Sing{f_{\gamma}})
\end{align*}
is contained in the right hand side of (\ref{main2cont}). 
\\\\
\noindent (\textbf{Case 2}) Next, we assume that 
$P_{\gamma(P)}(w^{0})= 0$ and $Q_{\gamma(Q)}(w^{0})\neq 0$.
In this case, we have $e=e_{P}>d^{\nu}_{P}$ and 
$e=e_{Q}=d_{Q}^{\nu}$ and hence
\[l=d^{\nu}_{P}<d^{\nu}_{Q}.\] 
Moreover by $d_f^{\nu} \not= 0$ the face $\gamma$ is 
of type II. 
Therefore, for $1\leq j\leq k$
we have
\[
\ov{A_{j}}=Q_{\gamma(Q)}(w^0)\dfrac{\partial 
P_{\gamma(P)}}{\partial z_{j}}(w^{0}).
\]
Since $P_{\gamma(P)}(w^0)=0$ and $\gamma$ is of type~II,
by the non-degeneracy condition (Definition~\ref{nondeg}),
$\frac{\partial P_{\gamma(P)}}{\partial z_{j}}(w^0)\neq 0$ 
for some $1\leq j\leq k$.
Hence, $J$ is not empty.
On the other hand, in this case we have
\begin{align*}
(\mbox{RHS\ of (\ref{iden5})})=Q_{\gamma(Q)}(w^{0})
d_{P}^{\nu}P_{\gamma(P)}(w^0)=0.
\end{align*}
But, this contradicts the second assertion 
of Lemma~\ref{NZlem5}.
\\\\
\noindent (\textbf{Case 3}) Similarly, 
we assume that $P_{\gamma(P)}(w^{0})\neq  0$ and $Q_{\gamma(Q)}(w^{0})= 0$.
In this case, we have $e=e_{P}=d^{\nu}_{P}$ 
and $e=e_{Q}>d_{Q}^{\nu}$ and hence
\[l=d^{\nu}_{Q}<d^{\nu}_{P}.\]
Moreover by $d_f^{\nu} \not= 0$ the face $\gamma$ is 
of type II. 
Therefore, for $1\leq j\leq k$
we have
\[\ov{A_{j}}=-P_{\gamma(P)}(w^{0})\dfrac{\partial 
Q_{\gamma(Q)}}{\partial z_{j}}(w^0).\]
Since $Q_{\gamma(Q)}(w^0)=0$ and $\gamma$ is of type~II,
by the non-degeneracy condition,
$\frac{\partial Q_{\gamma(Q)}}{\partial z_{j}}(w^0)
\neq 0$ for some $1\leq j\leq k$.
Hence, $J$ is not empty.
On the other hand, we have
\[
(\mbox{RHS of (\ref{iden5})})=-P_{\gamma(P)}(
w_{0})d_{Q}^{\nu}Q_{\gamma(Q)}(w^0)=0.
\]
But, this contradicts the second assertion of Lemma~\ref{NZlem5}.
\\\\
\noindent (\textbf{Case 4}) Finally, we 
assume that $P_{\gamma(P)}(w^{0})=  0$ and $Q_{\gamma(Q)}(w^{0})= 0$.
In this case, we have $e=e_{P}>d^{\nu}_{P}$ and $e=e_{Q}>d^{\nu}_{Q}$. 
Since $\nu \in \RR^n \setminus \RR^n_{\geq 0}$ the 
face $\gamma$ is of type I or II. Then 
by $P_{\gamma(P)}(w^{0})=  0$, $Q_{\gamma(Q)}(w^{0})= 0$ and 
the non-degeneracy condition, the complex vectors 
$\grad{P_{\gamma(P)}}(w^0)$ and $\grad{Q_{\gamma(Q)}}(w^{0})$
are linearly independent.
Therefore, by \eqref{key} we get $J\neq \emptyset$. 
On the other hand,
we have
\[(\mbox{RHS of (\ref{iden5})})=\wt{Q_{e}}
d_{P}^{\nu}P_{\gamma(P)}(w^{0})-\wt{P_{e}}
d_{Q}^{\nu}Q_{\gamma(Q)}(w^{0})=0.
\]
But, this contradicts the second assertion of Lemma~\ref{NZlem5}.

This completes the proof.
\end{proof}

 Combining Theorems \ref{NST} and \ref{main2}, 
we obtain Theorem \ref{MTM}. We will 
consider the following condition:
\begin{align}\tag{$*$}
\mbox{For any vector $u\in \RR^n\setminus 
\RR^n_{\geq 0}$, we have $d^{u}_{Q}\geq d^{u}_{P}$.}
\end{align}
It is satisfied if $P(0) \not= 0$, 
$Q(0) \not= 0$ and $\Nt(Q) \subset \Nt(P)$. 
This is the case in particular 
when $Q(z)=1$ (i.e. $f(z)=P(z)$ 
is a polynomial) and $P(0)=f(0) \not= 0$. 

\begin{thm}\label{conthm} 
In the situation in Theorem~\ref{MTM},
assume moreover the condition ($*$).
Then we have
\[\Bif_{f}\subset f(\Sing{f})\cup \Bigl(
\bigcup_{\gamma\in \mathscr{F}_{\mathrm{I}}}
f_{\gamma}(\Sing{f_{\gamma}}) \Bigr).\] 
\end{thm}

\begin{proof}
Assume that a point $s_0 \in \Sim_{f}\setminus f(\Sing{f})$ 
is not contained in $\cup_{\gamma\in \mathscr{F}_{\mathrm{I}}}
f_{\gamma}(\Sing{f_{\gamma}})$. 
It is enough to get a contradiction only for $s_0=0$. 
Let us assume $s_0=0$. 
We will use the notations and the results before 
($\bigstar$) in the proof of Theorem~\ref{main2}.
Then, we have $e_{P}>e_{Q}$.
Therefore, if $P_{\gamma(P)}(w^{0})\neq 0$, 
we have $e_{P}=d_{P}^{\nu}$ and hence 
$d_{P}^{\nu}>e_{Q}\geq d^{\nu}_{Q}$,
which contradicts the condition ($*$). 
Therefore, we have 
\begin{align*}
P_{\gamma(P)}(w^{0})=0. 
\end{align*}
By the condition ($*$),
for $1\leq j\leq k$ the degree of the $j$-th 
entry of the left hand side of (\ref{iden4}) 
is larger than or equal to $e_{Q}+d_{P}^{\nu}-\nu_{j}$. 
Let $A_j \in \CC$ be the 
coefficient of $t^{e_{Q}+d_{P}^{\nu}-\nu_{j}}$ in it. 
Then its complex conjugate $\ov{A_{j}}$ 
is expressed as 
\[\ov{A_{j}}=Q_{e_{P}}\dfrac{\partial 
P_{\gamma(P)}}{\partial z_{j}}(w^{0}).\]
We define $J$ and $j_{0}$ as (\ref{ai}) 
and (\ref{jzero}). 
Since $\nu \in \RR^n \setminus \RR^n_{\geq 0}$ the 
face $\gamma$ is of type I or II. 
If $\gamma$ is of type I, $Q_{\gamma(Q)}(w^0) \not= 0$ 
and $J= \emptyset$, we have $w^0 \in \Sing{f_{\gamma}}$ 
and 
\begin{align*}
s_0 =0 =f_{\gamma}(w^0) \in 
f_{\gamma}(\Sing{f_{\gamma}}). 
\end{align*}
This is a contradiction. So, in the case where 
$\gamma$ is of type I and $Q_{\gamma(Q)}(w^0) \not= 0$, 
we have $J \not= \emptyset$. Also in the other 
cases (where $\gamma$ is of type~II or 
$P_{\gamma(P)}(w^{0})=Q_{\gamma(Q)}(w^{0})=0$), 
by $P_{\gamma(P)}(w^{0})= 0$ and 
the non-degeneracy condition 
we have $\frac{\partial 
P_{\gamma(P)}}{\partial z_{j}}(w^{0})\neq 0$ 
for some $1 \leq j \leq k$ and hence $J\neq \emptyset$.
Similarly to the argument in the proof of Theorem~\ref{main2}, 
by using $e_Q \geq d_Q^{\nu} \geq d_P^{\nu}$ 
we obtain $\nu_{j}=\nu_{j_{0}}$ for any 
$j\in J$ and $\nu_{j_{0}}<0$.
Moreover,
in this situation, 
we have an equality similar to (\ref{iden5}):
\begin{align*}
Q_{e_{Q}}\nu_{j_{0}}\sum_{j\in J}|w_{j}^{0}|^2=
d_{P}^{\nu}P_{\gamma(P)}(w^{0}).
\end{align*}
The right hand side is $0$.
Since the left hand side is not zero,
this is a contradiction.
\end{proof}

\begin{cor} (N\'emethi and Zaharia \cite[Theorem 2]{N-Z1}) 
In the situation in Theorem~\ref{MTM}, 
assume moreover that 
$Q(z)=1$ (i.e. $f(z)=P(z)$ 
is a polynomial) and $P(0)=f(0) \not= 0$. 
Then we have
\[\Bif_{f}\subset f(\Sing{f})\cup \Bigl(
\bigcup_{\gamma\in \mathscr{F}_{\mathrm{I}}}
f_{\gamma}(\Sing{f_{\gamma}}) \Bigr).\] 
\end{cor}

In this corollary, for the face $\gamma = \{ 0 \} 
\prec \Nt(f)$ of type I we have 
$\gamma(P)= \gamma(Q)= \{ 0 \}$, 
$f_{\gamma}(z)=f(0) \not= 0$ and 
\begin{align*}
f_{\gamma}(\Sing{f_{\gamma}}) = 
\{ f(0) \}. 
\end{align*}

\section{The two dimensional case and examples}\label{2-dim case}

In this section, we show that in the two dimensional case $n=2$ 
the inclusion 
\begin{align*}
\Bif_{f}\subset f(\Sing{f})\cup \{0\}\cup 
\Bigl( \bigcup_{\gamma\in 
\mathscr{F}_{\mathrm{I}}}f_{\gamma}(\Sing{f_{\gamma}}) \Bigr)
\end{align*}
in Theorem \ref{MTM} is indeed an equality outside a finite 
subset of $\CC$ and give some examples. Let 
$\gamma \prec \Nt(f)$ be a $0$-dimensional face of type I. Then 
$\gamma (P) \prec \Nt(P)$ and $\gamma (Q) \prec \Nt(Q)$ are 
also $0$-dimensional, $\gamma (P)= \gamma (Q)$ and 
\begin{align*}
f_{\gamma}=
\frac{P_{\gamma (P)}}{Q_{\gamma (Q)}}: 
T \setminus Q_{\gamma (Q)}^{-1}(0) \longrightarrow \CC 
\end{align*}
is a non-zero constant function on $T \setminus Q_{\gamma (Q)}^{-1}(0)
=T$ (here $Q_{\gamma (Q)}$ is a monomial). 
We denote its value by $c( \gamma ) \in \CC$. Then we define 
a subset $\CF_f \subset \CC$ by 
\begin{align*}
\CF_f:= \{ c( \gamma ) \in \CC \ | \ 
\gamma \in \mathscr{F}_{\mathrm{I}}, \ \dim \gamma =0 \} \subset \CC. 
\end{align*}

\begin{thm}\label{main-3}
In the situation of Theorem \ref{MTM}, assume moreover 
that $n=2$. 
Then we have an equality 
\begin{align}\label{main-3cont}
\Bif_{f} \setminus ( \{ 0 \} \cup \CF_f) = 
\Bigl\{ f(\Sing{f}) \cup 
\Bigl( \bigcup_{\gamma\in 
\mathscr{F}_{\mathrm{I}}}f_{\gamma}(\Sing{f_{\gamma}}) \Bigr) 
\Bigr\} \setminus ( \{ 0 \} \cup \CF_f).
\end{align}
\end{thm}

\begin{proof}
We follow the proof of \cite[Theorem 4.3]{T}. 
Since $f( {\rm Sing} f) \subset \Bif_f$, 
it suffices to show the inclusion 
\begin{align*}
\Bigl( \bigcup_{\gamma\in 
\mathscr{F}_{\mathrm{I}}}f_{\gamma}(\Sing{f_{\gamma}}) \Bigr) 
\setminus \left( f(\Sing{f}) \cup \{ 0 \} \cup \CF_f \right) 
\subset \Bif_{f}.
\end{align*}
Let $s_0 \in \CC$ be a point in the left hand side. 
We define a $\ZZ$-valued function 
$\chi_c: \CC \longrightarrow \ZZ$ on $\CC$ by 
\begin{equation*}
\chi_c(s)= \sum_{j \in \ZZ} (-1)^j \dim 
H^j_c ( f^{-1}(s); \CC ) \qquad (s \in \CC )
\end{equation*}
and its jump $E_f( \sigma ) \in \ZZ$ 
at $s_0 \in \CC$ by 
\begin{equation*}
E_f( s_0 )= - \left\{ \chi_c( s_0 + \varepsilon ) 
- \chi_c( s_0 ) \right\} \in \ZZ,  
\end{equation*}
where $\varepsilon >0$ is sufficiently small. 
Then it is enough to show that $E_f( s_0 ) \not= 0$. 
From now, we will use the terminologies in 
\cite{Dimca}, \cite{H-T-T} and \cite{K-S} etc. 
For the point $s_0 \in \CC$ define a 
function $h: \CC \longrightarrow 
\CC$ on $\CC$ by $h(s)=s- s_0$ so that 
we have $h^{-1}(0)= \{ s_0 \}$. Then 
we have 
\begin{equation*}
E_f( s_0 )= - 
\sum_{j \in \ZZ} (-1)^j \dim 
H^j \phi_h(R f_! \CC_{\CC^2 \setminus Q^{-1}(0)})_{s_0},  
\end{equation*}
where $\phi_h:  \mathrm{D^{b}_{c}} (\CC) \longrightarrow 
\mathrm{D^{b}_{c}} ( \{ s_0 \} )$ is Deligne's vanishing cycle 
functor associated to $h$. 
Now we introduce an equivalence relation $\sim$ on 
(the dual vector space of) $\RR^2$ by 
$u \sim u^{\prime} \Longleftrightarrow 
\gamma_f^u = \gamma_f^{u^{\prime}}$. We can easily 
see that for any face $\gamma \prec \Nt(f)$ 
of $\Nt(f)$ the closure of the 
equivalence class associated to it in $\RR^2$ 
is an $(2- \dim \gamma )$-dimensional rational 
convex polyhedral cone $\sigma (\gamma )$ in $\RR^2$. Moreover 
the family $\{ \sigma (\gamma ) \ | \ 
\gamma \prec  \Nt(f) \}$ of cones in $\RR^2$ 
thus obtained is a subdivision of $\RR^2$. 
We call it the dual subdivision of $\RR^2$ by 
$\Nt(f)$. If $\dim \Nt(f)=2$ it 
satisfies the axiom of fans (see \cite{Fulton} and 
\cite{Oda} etc.). We call it the dual fan of 
$\Nt(f)$. Let $\Sigma_0$ be a complete fan in $\RR^2$ 
obtained by subdividing the 
dual subdivision. Note that all the cones in it 
are proper and convex. 
Let $\Sigma$ be a smooth and complete fan in $\RR^2$ 
containing all the $1$-dimensional cones $\tau 
\simeq \RR^1_{\geq 0}$ in $\Sigma_0$ 
such that $\tau \cap \RR^2_{\geq 0} = \{ 0 \}$ and 
satisfying the condition 
$\RR^2_{\geq 0} \in \Sigma$. Let 
$X_{\Sigma}$ be the toric variety 
associated to it. Then $X_{\Sigma}$ is a smooth 
compactification of $\CC^2$. 
This construction of $X_{\Sigma}$ 
is inspired from the one in Zaharia \cite{Zaharia}. 
Recall that the torus $T=( \CC^*)^2$ acts on 
$X_{\Sigma}$ and the $T$-orbits in it are 
parametrized by the cones $\tau$ in $\Sigma$. 
For a cone $\tau \in \Sigma$ denote 
by $T_{\tau} \simeq (\CC^*)^{2-\dim 
\tau}$ the corresponding $T$-orbit. If 
$\tau \in \Sigma$ is not contained in $\RR^2_{\geq 0}$ 
and its relative interior is contained in that of 
the cone $\sigma (\gamma )$ for a type II face $\gamma$ of 
$\Nt (f)$, then by the non-degeneracy condition 
the closures $\overline{P^{-1}(0)}, \overline{Q^{-1}(0)} 
\subset X_{\Sigma}$ 
of $P^{-1}(0), Q^{-1}(0) \subset \CC^2$ 
respectively in $X_{\Sigma}$  
intersect $T_{\tau}$ transversally. 
At such intersection points, (the meromorphic 
extension) of $f$ to $X_{\Sigma}$ may have indeterminacy. 
Moreover for $n=2$ we have 
\begin{align*}
(\overline{P^{-1}(0)} \cap T_{\tau}) \cap 
(\overline{Q^{-1}(0)} \cap T_{\tau}) = \emptyset. 
\end{align*}
If $\tau \in \Sigma$ is not contained in $\RR^2_{\geq 0}$ 
and its relative interior is contained in that of 
the cone $\sigma (\gamma )$ for a type I face $\gamma$ of 
$\Nt (f)$ such that $\dim \gamma =1$, 
then the order of the meromorphic extension 
of $f$ to $X_{\Sigma}$ along the $T$-divisor 
$\overline{T_{\tau}} \subset X_{\Sigma}$ 
is zero. Moreover, by the non-degeneracy condition we have 
\begin{align*}
(\overline{P^{-1}(0)} \cap T_{\tau}) \cap 
(\overline{Q^{-1}(0)} \cap T_{\tau}) = \emptyset. 
\end{align*}
As in \cite[Section 3]{T}, 
by constructing a tower of blow-ups $\pi : \widetilde{X_{\Sigma}} 
\longrightarrow X_{\Sigma}$ of $X_{\Sigma}$ 
to eliminate the indeterminacy of $f$ we 
obtain a commutative diagram:
\begin{equation*}
\begin{CD}
\CC^2  \setminus Q^{-1}(0) @>{\iota}>> \widetilde{X_{\Sigma}} 
\\
@V{f}VV   @VV{g}V
\\
\CC  @>>{j}>  \PP^1 
\end{CD}
\end{equation*}
of holomorphic maps, 
where $\iota : \CC^2  \setminus Q^{-1}(0) 
\hookrightarrow \widetilde{X_{\Sigma}}$ and 
$j : \CC  
\hookrightarrow \PP^1$ are 
the inclusion maps and 
$g$ is proper. By this construction, if 
$\tau \in \Sigma$ is not contained in $\RR^2_{\geq 0}$ 
and its relative interior is contained in that of 
the cone $\sigma (\gamma )$ for a type I face $\gamma$ of 
$\Nt (f)$, then $\pi$ induced an isomorphism 
$\pi^{-1}(T_{\tau}) \simeq T_{\tau}$. So we 
regard $T_{\tau}$ as a subset of 
$\widetilde{X_{\Sigma}}$. Since $g$ is proper, 
by \cite[Proposition 4.2.11]{Dimca} and 
\cite[Exercise VIII.15]{K-S} 
we thus obtain an isomorphism 
\begin{equation*}
\phi_h(R  f_! \CC_{\CC^2 \setminus Q^{-1}(0)})_{s_0} \simeq 
R \Gamma (g^{-1}( s_0 ); \phi_{h \circ g}
( \iota_! \CC_{\CC^2 \setminus Q^{-1}(0)})). 
\end{equation*}
By our choice of the point $s_0 \in \CC$, the support 
of $\phi_{h \circ g}( \iota_! \CC_{\CC^2 \setminus Q^{-1}(0)}) 
\in \mathrm{D^{b}_{c}}  (  g^{-1}( s_0 ) )$ 
is contained in the (non-empty) finite subset of 
$g^{-1}( s_0 ) \subset \tl{X_{\Sigma}}$ 
consisting of the points 
$q \in T_{\sigma (\gamma)}$ for 
$1$-dimensional type I faces $\gamma$ of 
$\Nt (f)$ such that $q \in \Sing{f_{\gamma}}$ 
and $s_0 = f_{\gamma}(q)$. 
Here we naturally regard $f_{\gamma}$ as 
a rational function on 
$T_{\sigma (\gamma)} \simeq \CC^*$. In 
a neighborhood of the point 
$q \in T_{\sigma (\gamma)}$ 
it coincides with the restriction of $g$ to 
$T_{\sigma (\gamma)} \subset 
\widetilde{X_{\Sigma}}$. For one 
$q \in T_{\sigma (\gamma)}$ of such points,  
let $\mu_q \geq 0$ be 
the Milnor number of the 
 (possibly singular) complex hypersurface 
 $g^{-1}( s_0 )$ (in fact, it is an 
algebraic curve having at most 
an isolated singular point at $q$) of 
$\widetilde{X_{\Sigma}}$ at $q$. Denote by 
$m_q \geq 2$ the multiplicity of the zeros of 
the function $f_{\gamma}- s_0$ at $q$. Note that 
in a neighborhood of the point $q$ in 
$\widetilde{X_{\Sigma}}$ the sequence 
\begin{align*}
0  \rightarrow \CC_{\CC^2 \setminus Q^{-1}(0)} \rightarrow 
\CC_{\widetilde{X_{\Sigma}}} \rightarrow 
\CC_{T_{\sigma (\gamma)}} \rightarrow 0
\end{align*}
is exact. Then as in the final part 
of the proof of \cite[Theorem 4.3]{T} we obtain 
\begin{align*}
\chi ( \phi_{h \circ g}( \iota_! 
\CC_{\CC^2 \setminus Q^{-1}(0)})_q ) = 
- \mu_q -(m_q-1) <0. 
\end{align*}
Consequently, we get $E_f( s_0 )>0$. 
This completes the proof. 
\end{proof}

By Theorems \ref{conthm} and \ref{main-3} we 
obtain the following result. 

\begin{cor} 
In the situation of Theorem \ref{MTM}, assume moreover 
the condition ($*$) and that $n=2$. 
Then we have an equality 
\begin{align}\label{main-3cont}
\Bif_{f} \setminus  \CF_f = 
\Bigl\{ f(\Sing{f}) \cup 
\Bigl( \bigcup_{\gamma\in 
\mathscr{F}_{\mathrm{I}}}f_{\gamma}(\Sing{f_{\gamma}}) \Bigr) 
\Bigr\} \setminus \CF_f.
\end{align}
\end{cor}

Similarly, also in higher dimensions $n \geq 3$ we 
obtain results similar to the ones in \cite{T} 
and \cite{Zaharia}. 
We leave their precise formulations to the readers. 
If $Q(z)=1$ and $f(z)=P(z)$ 
is a polynomial which is non-degenerate 
(at infinity) and convenient, then by a result of
Broughton \cite{Broughton} 
the polynomial map $f: \CC^n \rightarrow \CC$
is tame at infinity and
\begin{align*}
\Bif_f = f( {\rm Sing} f).
\end{align*}
However, for rational functions $f(z)= \frac{P(z)}{Q(z)}$,
by Theorems \ref{MTM} and \ref{main-3}, 
even if $P(z)$ and $Q(z)$ are convenient 
there might be some type I faces of $\Nt(f)$ and 
hence we do not have the equality 
$\Bif_f = f( {\rm Sing} f)$ in general.

For the value $0$, let us consider the following example.

\begin{ex}\label{ex1}{\rm   Let $f=
\frac{x^2+y}{x+y}.$ It is easy to check that
$f$ is non-degenerate in the sense of Definition \ref{nondeg}.
Let us consider the value $0\in \CC$. For a small disc
$D\subset \CC$ centered at it, we have
$$f^{-1}(D)=\left\{(x, \frac{x^2-tx}{1-t})\ | \ x
\in \CC\setminus \{0, 1\}, t\in D\right\}.$$
It is easy to check that the restriction map
$f: f^{-1}(D)\to D$ is a  trivial fibration.
This means $0\notin \Bif_f.$

Moreover, by \cite[Theorem 1.2]{Thang} we have:
$$\Bif_f= \Bif_{\infty}\cup f(\Sing f)\cup \mathrm{K}_1(f),$$
see \cite[Definition 2.2]{Thang} for the
definition of the set $\mathrm{K}_1(f)$, while
$\Bif_{\infty}(f)$ is the set of critical
value at infinity of $f$. One can easily
check that in this example $f(\Sing f)
= \mathrm{K}_1(f)=\emptyset$. Therefore $\Bif_f=\Bif_{\infty}(f)$.
Let us consider the polynomial $$g_t(x,y):= x^2+y- t(x+y)$$
 and $\delta(y,t)$ to be the discriminant
of $g_t(x,y)$ with respect to the variable $x$. Then
$$\delta(y,t)=4(1-t)y-t^2.$$
Hence by \cite[Corollary 3.7]{Thang} we
get $\Bif_f=\Bif_{\infty}(f)=\{1\}.$

On the other hand, for the set on the
right hand side of the inclusion
(\ref{th11inc}) in Theorem~\ref{MTM},
the only non-empty set among those
of $\bigcup_{\gamma\in
\mathscr{F}_{\mathrm{I}}}
f_{\gamma}(\Sing{f_{\gamma}})$ comes
from the face function $\frac{y}{y}$ which
again provides us the value $1.$ 
}
\end{ex}

Regarding the set $\CF_f$, we will see from
the example below that
in general $\CF_f$ is not a subset of $\Bif_f$.

\begin{ex}{\rm
Let $f:= \frac{x+y}{x+2y}$. Then
$\CF_f=\{1/2, 1\}$. For any small
neighborhood $D$ of $1/2$ (such as
$D$ contains $1/2$ and does not contain $1$), we have
$f^{-1}(D)= \{( - \frac{1-2t}{1-t}y,y) :
y\in \CC^{*}\}$. Hence the restriction
$f: f^{-1}(D)\to D$ is a locally trivial
fibration. This means $1/2\notin \Bif_f$.
Similarly $1\notin \Bif_f$.
}
\end{ex}

\subsection*{Acknowledgments}
We thank the anonymous referee for his/her helpful comments.
The first author is funded by Mathematics Development Program - Vietnam Ministry of Education and Training under grant number B2020 - SPH - 03 CTTH.
The second author is supported by JSPS KAKENHI Grant Number 20J00922.

\bibliography{refrat}

\end{document}